\documentclass[12pt,a4wide]{article}
\usepackage[leqno]{amsmath}
\usepackage{graphics}
\usepackage{epsfig}
\usepackage[mathscr]{euscript}
\usepackage{mathrsfs}
\usepackage{amsmath,amsthm,amsfonts}
\usepackage{amssymb}
\usepackage{enumerate}

\newtheorem{lemma}{Lemma}[section]
\newtheorem{theorem}[lemma]{Theorem}
\newtheorem{prop}[lemma]{Proposition}
\newtheorem{claim}{Claim}

\newtheorem{cor}[lemma]{Corollary}
\theoremstyle{definition}

\newtheorem{defi}[lemma]{Definition}
\newtheorem{exa}[lemma]{Example}

\newcommand{\A}{\mathsf{A}}
\newcommand{\D}{\mathsf{D}}
\newcommand{\E}{\mathsf{E}}

\newcommand{\zz}{\mathbb{Z}}
\newcommand{\R}{\mathbf{R}}
\newcommand{\rr}{\mathbf{R}}

\newcommand{\TS}{\mathcal{S}}
\newcommand{\Ho}{\mathfrak{H}}
\newcommand{\Go}{\mathfrak{G}}

\newcommand{\sat}[1]{$#1$-saturated}

\newcommand{\h}{\mathscr{H}}
\DeclareMathOperator{\reduced}{red}
\newcommand{\redlam}{\Lambda^{\reduced}}
\DeclareMathOperator{\degr}{deg}

\DeclareMathOperator{\sgn}{sgn}

\newcommand{\nexteq}{\displaybreak[0]\\ &=}

\newcommand{\subgg}[2]{\langle\langle #1\rangle\rangle_{#2}}

\newcommand{\be}{\mathbf{e}}
\newcommand{\bu}{\boldsymbol{u}}
\newcommand{\bj}{\boldsymbol{j}}

\title{On fat Hoffman graphs with smallest eigenvalue at least $-3$}

\author{
Hye Jin Jang, Jack Koolen,\\
Akihiro Munemasa and Tetsuji Taniguchi}

\date{September 25, 2012}

\begin{document}
\maketitle

\begin{abstract}
We investigate fat Hoffman graphs with smallest
eigenvalue at least $-3$, using their special graphs. We show
that the special graph $\TS(\Ho)$
of an indecomposable fat Hoffman graph $\Ho$
is represented by the standard lattice or an irreducible root lattice.
Moreover, we show that if the special graph admits an integral
representation, that is, the lattice spanned by it is not an
exceptional root lattice, then the 
special graph $\TS^-(\Ho)$ is isomorphic to one of the
Dynkin graphs $\A_n$, $\D_n$, or extended Dynkin graphs
$\tilde{\A}_n$ or $\tilde{\D}_n$.
\end{abstract}

\section{Introduction}
Throughout this paper, a graph will mean an undirected graph without
loops or multiple edges.

Hoffman graphs were introduced by
Woo and Neumaier \cite{hlg} to extend the results of
Hoffman \cite{hoffman0}. 
Hoffman proved what we would call Hoffman's
limit theorem (Theorem~\ref{lim2}) which asserts that,
in the language of Hoffman graphs,
the smallest eigenvalue of a fat Hoffman graph is a limit
of the smallest eigenvalues of a sequence of ordinary graphs
with increasing minimum degree.
Woo and Neumaier \cite{hlg} 
gave a complete list of fat indecomposable
Hoffman graphs with smallest eigenvalue at least $-1-\sqrt{2}$.
From their list, we find that only
$-1,-2$ and $-1-\sqrt{2}$ appear as the
smallest eigenvalues.
This implies, in particular, that
$-1,-2$ and $-1-\sqrt{2}$ are limit points of the smallest
eigenvalues of a sequence of ordinary graphs
with increasing minimum degree. 
It turns out that
there are no others between $-1$ and $-1-\sqrt{2}$.
More precisely, for a negative real number $\lambda$,
consider the sequences
\begin{align*}
\eta^{(\lambda)}_k&=
\inf\{\lambda_{\min}(\Gamma)\mid 
\min\deg\Gamma\geq k,\;
\lambda_{\min}(\Gamma)>\lambda\}
\quad(k=1,2,\dots),\\
\theta^{(\lambda)}_k&=
\sup\{\lambda_{\min}(\Gamma)\mid 
\min\deg\Gamma\geq k,\;
\lambda_{\min}(\Gamma)<\lambda\}
\quad(k=1,2,\dots),
\end{align*}
where $\Gamma$ runs through graphs satisfying the conditions
specified above, namely, $\Gamma$ has minimum degree at least $k$
and $\Gamma$ has smallest eigenvalue greater than (or less than)
$\lambda$.
Then Hoffman \cite{hoffman0} has shown that
\begin{align*}
\lim_{k\to\infty}\eta^{(-2)}_k&=-1,\quad
&\lim_{k\to\infty}\theta^{(-1)}_k&=-2,\displaybreak[0]\\
\lim_{k\to\infty}\eta^{(-1-\sqrt{2})}_k&=-2,\quad
&\lim_{k\to\infty}\theta^{(-2)}_k&=-1-\sqrt{2}.
\end{align*}
In other words, real numbers in $(-2,-1)$ and $(-1-\sqrt{2},-2)$
are ignorable if our concern is the smallest eigenvalues of
graphs with large minimum degree.
Woo and Neumaier \cite{hlg} went on to prove that
\[
\lim_{k\to\infty}\eta^{(\alpha)}_k=-1-\sqrt{2},
\]
where
$\alpha\approx-2.4812$ is a zero of the cubic polynomial
$x^3+2x^2-2x-2$.
Recently, Yu \cite{Yu} has shown that analogous results
for regular graphs also hold.

Woo and Neumaier \cite[Open Problem 4]{hlg}
raised the problem of classifying fat Hoffman graphs with smallest
eigenvalue at least $-3$. 
They also proposed a generalization of a concept of a line graph
based on a family of isomorphism classes of Hoffman graphs.
This enables one to define generalized line graphs in a very simple
manner. As we shall see in Proposition~\ref{prop:line},
the knowledge of 
\sat{\mu}\
indecomposable fat Hoffman graphs gives
a description of all fat Hoffman graphs with smallest eigenvalue
at least $\mu$. For $\mu=-3$,
this in turn should give some information on the limit points
of the smallest eigenvalues of a sequence of ordinary graphs
with increasing minimum degree. Also, using the generalized
concept of line graphs, we can expect to give a description of
all graphs with smallest eigenvalue at least $-3$ and sufficiently
large minimum degree.
Thus our ultimate goal is to classify 
\sat{(-3)}\ indecomposable fat
Hoffman graphs, as proposed by Woo and Neumaier \cite{hlg}.

The purpose of this paper is to
begin the first step of this classification, by determining
their special graphs for such Hoffman graphs having an integral
reduced representation. One of our main result is that
the special graph $\TS^-(\Ho)$ of such a Hoffman graph $\Ho$
is isomorphic to one of the Dynkin
graphs $\A_n$, $\D_n$, or extended Dynkin graphs $\tilde{\A}_n$ or $\tilde{\D}_n$.
We also show that, even when the Hoffman graph $\Ho$
does not admit an integral representation, 
its special graph $\TS(\Ho)$ can be represented by one of the exceptional
root lattices ${\E}_n$ ($n=6,7,8$). 
This might mean that the rest of the work can be completed by
a computer as in the classification of maximal exceptional graphs
(see \cite{new}). Indeed, if one attaches a fat neighbor to every
slim vertex of an ordinary maximal exceptional graph, then we
obtain a $(-3)$-indecomposable fat Hoffman graph. However, 
maximal graphs among 
$(-3)$-indecomposable fat Hoffman graphs represented in $\E_8$
are usually much larger (see Example~\ref{ex:mE8} and the comment
following it), so the problem is not as trivial as it looks. We plan to
discuss in the subsequent papers the
determination of these special graphs
and the corresponding Hoffman graphs.

\section{Hoffman graphs and their smallest eigenvalues}

\subsection{Basic theory of Hoffman graphs}
In this subsection we discuss the basic theory of Hoffman graphs. 
Hoffman graphs were introduced by Woo and Neumaier \cite{hlg},
and most of this section is due to them. Since the sums
of Hoffman graphs appear only implicitly in \cite{hlg} and
later formulated by Taniguchi \cite{paperI},
we will give proof of the results that use sums
for the convenience of the reader.

\begin{defi}\label{df:Hg}
A \emph{Hoffman graph} $\mathfrak{H}$ is a pair 
$(H,\mu)$ of a graph $H=(V,E)$ and a labeling map $\mu:V\to\{f,s\}$,
satisfying the following conditions:
\begin{enumerate}
\item \label{en:Hg1}
every vertex with label $f$ is adjacent to at least
one vertex with label $s$;
\item \label{en:Hg2}
vertices with label $f$ are pairwise non-adjacent.
\end{enumerate}
We call a vertex with label $s$ a \emph{slim vertex}, and
a vertex with label $f$ a \emph{fat vertex}. We denote by
$V_s = V_s(\Ho)$ (resp.\ $V_f(\Ho)$) the set of slim (resp.\ fat)
vertices of $\Ho$. The subgraph of a
Hoffman graph $\Ho$ induced on $V_s(\Ho)$ is called the \emph{slim
subgraph} of $\Ho$. If every slim vertex of a Hoffman graph 
$\mathfrak{H}$ has a fat neighbor, then we call $\mathfrak{H}$ \emph{fat}.

For a vertex $x$ of $\Ho$ we define $N^f(x) = N^f_{\Ho}(x)$ (resp.\
$N^s(x) = N^s_{\Ho}(x)$) the 
set of fat (resp.\ slim) neighbors of $x$ in $\Ho$.
The set of all
neighbors of $x$ is denoted by $N(x) = N_{\Ho}(x)$,
	that is $N(x) = N^f(x) \cup N^s(x)$.
In a similar fashion, 
for vertices $x$ and $y$
	we define $N^f(x,y)=N^f_{\Ho}(x,y)$ to be the 
set 
of common fat neighbors of $x$ and $y$.
\end{defi}

\begin{defi}\label{df:sub}
A Hoffman graph $\Ho_1 = (H_1, \mu_1)$ is called
	an (induced) \emph{Hoffman subgraph} of another Hoffman graph $\Ho= (H, \mu)$,
	if $H_1$ is an (induced) subgraph of $H$ and 
$\mu(x) = \mu_1(x)$ for all vertices $x$ of $\Ho_1$.
Unless stated otherwise, by a subgraph we mean an induced Hoffman subgraph.
For a subset $S$ of $V_s(\Ho)$, we denote by $\subgg{S}{\Ho}$ the 
subgraph of $\Ho$ induced on the set of vertices
\[
S\cup\left(\bigcup_{x\in S}N_\Ho^f(x)\right). 
\]
\end{defi}

\begin{defi}\label{df:iso}
Two Hoffman graphs $\Ho= (H, \mu)$ and $\Ho'= (H', \mu')$
	are called \emph{isomorphic}
	if there exists an isomorphism from $H$ to $H'$ which preserves the labeling.
\end{defi}
 An ordinary graph $H$ without labeling can be
regarded as a Hoffman graph $\Ho= (H, \mu)$ without any fat vertex,
	i.e., $\mu(x) = s$ for all vertices $x$.
Such a
graph is called a \emph{slim graph}.


\begin{defi}\label{df:eigen}
For a Hoffman graph $\mathfrak{H}$, let $A$ be its adjacency matrix,
\[
A=
\begin{pmatrix}
A_s & C \\
C^T & O
\end{pmatrix}
\]
in a labeling in which the fat vertices come last.
\emph{Eigenvalues} of $\Ho$ are the eigenvalues of the
real symmetric matrix $B(\Ho)=A_s-CC^T$.
Let $\lambda_{\min}(\Ho)$ denote the smallest eigenvalue of $\Ho$.
\end{defi}

\begin{defi}[{\cite{hlg}}]\label{df:rep}
For a Hoffman graph $\Ho$ and a positive integer $n$, 
a mapping $\phi:\ V(\Ho)\to \R^n$
such that
\[
(\phi (x),\phi (y))=
\begin{cases}
m & \text{if $x=y\in V_s(\Ho)$,}\\
1 & \text{if $x=y\in V_f(\Ho)$,}\\
1 & \text{if $x$ and $y$ are adjacent in $\Ho$,}\\
0 & \text{otherwise,}
\end{cases}
\] 
is called a \emph{representation of norm $m$}.
We denote by $\Lambda(\Ho,m)$
the lattice generated by $\{ \phi(x) \mid x \in V(\Ho)\}$.
Note that the isomorphism class of $\Lambda(\Ho,m)$ depends only
on $m$, and is independent of $\phi$, justifying the notation.
\end{defi}

\begin{defi}\label{df:redrep}
For a Hoffman graph $\Ho$ and a positive integer $n$, 
a mapping $\psi:\ V_s(\Ho)\to \R^n$
such that
\[
 (\psi (x),\psi (y))=
\begin{cases}
m - |N^f_{\Ho}(x)| & 
\text{if $x=y$,}\\
1 - |N^f_{\Ho}(x,y)| & 
\text{if $x$ and $y$ are adjacent,}\\
-|N^f_{\Ho}(x,y)| & \text{otherwise.}\\
\end{cases}
\]
is called a \emph{reduced representation of norm $m$}.
We denote by $\redlam(\Ho,m)$
the lattice generated by $\{ \psi(x) \mid x \in V_s(\Ho)\}$.
Note that the isomorphism class of $\redlam(\Ho,m)$ depends only
on $m$, and is independent of $\psi$, justifying the notation.
\end{defi} 

While it is clear that a representation of norm $m>1$ is an injective
mapping, a reduced representation of norm $m$ may not be. 
See Section~\ref{sec:int} for more details.

\begin{lemma}\label{lm:WN0}
Let $\Ho$ be a Hoffman graph
having a representation of norm $m$.
Then $\Ho$ has a reduced representation of norm $m$, and
$\Lambda (\Ho,m)$ is isomorphic to
$\redlam (\Ho,m)\oplus \zz^{|V_f(\Ho)|}$ as a lattice.
\end{lemma}
\begin{proof}
Let $\phi:V(\Ho)\to \rr^n$ be a representation of norm $m$.
Let $U$ be the subspace of $\rr^n$
	generated by $\{\phi (x)\mid x\in V_f(\Ho)\}$.
Let $\rho$, $\rho^{\bot}$ denote the orthogonal projections
	from $\rr^n$ onto $U$, $U^{\bot}$, respectively.
Then $\rho^\perp\circ\phi$ is a reduced representation of norm $m$,
$\rho^{\bot}(\Lambda (\Ho,m))\cong\redlam (\Ho,m)$, and
$\rho (\Lambda (\Ho,m))\cong \zz^{|V_f(\Ho)|}$.
\end{proof}


\begin{theorem}\label{thm:WN1}
For a Hoffman graph $\Ho$,
	the following conditions are equivalent:
	\begin{enumerate}
	\item $\Ho$ has a representation of norm $m$,
	\item $\Ho$ has a reduced representation of norm $m$,
	\item $\lambda_{\min}(\Ho)\ge -m$.
	\end{enumerate}
\end{theorem}
\begin{proof}
From Lemma~\ref{lm:WN0},
	we see that (i) implies (ii).

Let $\psi$ be a reduced representation of $\Ho$ of norm $m$.
Then the matrix $B(\Ho)+mI$ is the Gram matrix of the
image of $\psi$, and hence positive semidefinite.
This implies that
$B(\Ho)$ has smallest eigenvalue at least $-m$
and hence $\lambda_{\min}(\Ho) \geq -m$. 
This proves (ii) $\implies$ (iii).

The proof of equivalence of (i) and (iii) can be found
in \cite[Theorem~3.2]{hlg}.
\end{proof}

From Theorem~\ref{thm:WN1},
we obtain the following lemma:

\begin{lemma}\label{lem:eig}
If $\Go$ is a subgraph of a Hoffman graph $\Ho$,
then $\lambda_{\min}(\Go)\geq\lambda_{\min}(\Ho)$ holds.
\end{lemma}
\begin{proof} 
Let $m := -\lambda_{\min}(\Ho)$. Then $\Ho$ has a representation $\phi$ of
norm $m$ by Theorem~\ref{thm:WN1}. Restricting $\phi$ to the vertices of
$\Go$ we obtain a representation of norm $m$ of $\Go$, which implies 
$\lambda_{\min}(\Go) \geq -m$ by Theorem~\ref{thm:WN1}.
\end{proof}

In particular, if $\Gamma$ is the slim
subgraph of $\Ho$, then $\lambda_{\min}(\Gamma)\geq
\lambda_{\min}(\Ho)$.

Under a certain condition, equality holds in Lemma~\ref{lem:eig}.
To state this condition we need to introduce decompositions
of Hoffman graphs. This was formulated first by the third author
\cite{paperI}, although it was already implicit in \cite{hlg}.

\begin{defi}\label{df:sum}
Let $\Ho$ be a Hoffman graph,
	and let $\Ho^1$ and $\Ho^2$ be two non-empty induced Hoffman subgraphs of $\Ho$.
	The Hoffman graph $\Ho$ is said to be the {\em sum}\/ of $\Ho^1$ and $\Ho^2$,
	written as $\Ho = \Ho^1 \uplus \Ho^2$,
	if the following conditions are satisfied:
\begin{enumerate}
\item $V(\Ho) = V(\Ho^1) \cup V(\Ho^2)$;
\item $\{V_s(\Ho^1), V_s(\Ho^2)\}$ is a partition of $V_s(\Ho)$;
\item if $x\in V_s(\Ho^i)$, $y\in V_f(\Ho)$ and $x\sim y$, then $y\in V_f(\Ho^i)$;
\item if $x\in V_s(\Ho^1)$, $y\in V_s(\Ho^2)$, 
then $x$ and $y$ have at most one common fat neighbor,
and they have one if and only if they are adjacent.
\end{enumerate}
If $\Ho=\Ho^1\uplus \Ho^2$ for some non-empty subgraphs
	$\Ho^1$ and $\Ho^2$,
	then we call $\Ho$ {\em decomposable}.
Otherwise $\Ho$ is called {\em indecomposable}.
Clearly,
	a disconnected Hoffman graph is decomposable.
\end{defi}

It follows easily that the above-defined sum is associative and
so that the sum
\[
\Ho=\biguplus_{i=1}^n \Ho^i
\]
is well-defined.
The main reason for defining the sum of Hoffman graphs is the following lemma. 

\begin{lemma}\label{lem:decomp}
Let $\Ho$ be a Hoffman graph
	and let $\Ho^1$ and $\Ho^2$ be two (non-empty) induced Hoffman subgraphs
	of $\Ho$ satisfying {\rm(i)--(iii)} of Definition~\ref{df:sum}.
Let $\psi$ be a reduced representation of $\Ho$ of norm $m$.
Then the following are equivalent.
\begin{enumerate}
\item
$\Ho=\Ho^1\uplus \Ho^2$,
\item
$(\psi(x), \psi(y)) = 0$ for any 
$x \in V_s(\Ho^1)$ and $y \in V_s(\Ho^2)$.
\end{enumerate}
\end{lemma}
\begin{proof}
This
follows easily from the definitions of $\Ho = \Ho^1 \uplus \Ho^2$
	and a reduced representation of norm $m$.
\end{proof}

\begin{lemma}\label{lem:sumeigen}
If $\Ho=\Ho^1\uplus \Ho^2$, then
\[
\lambda_{\min}(\Ho)=\min\{\lambda_{\min}(\Ho^1),
\lambda_{\min}(\Ho^2)\}.
\]
\end{lemma}
\begin{proof}
Let $m = - \min\{ \lambda_{\min}(\Ho^1),\lambda_{\min}(\Ho^2)\}$. 
In view of Lemma~\ref{lem:eig} we only need to show that 
$\lambda_{\min}(\Ho)\geq -m$.
By Theorem~\ref{thm:WN1}, 
$\Ho^i$ has a reduced representation $\psi^i:V(\Ho^i)\to\R^{n_i}$
of norm $m$, for each $i=1,2$.
Define $\psi:V(\Ho)\to\R^{n_1}\oplus\R^{n_2}$ by
$\psi(x)=\psi^1(x)\oplus0$ if $x\in V(\Ho^1)$,
$\psi(x)=0\oplus\psi^2(x)$ otherwise.
It is easy to check that 
$\psi$ is a reduced representation of norm $m$,
and the result then follows from Theorem~\ref{thm:WN1}.
\end{proof}

\subsection{Hoffman's limit theorem}
In this subsection, we state and prove
Hoffman's limit theorem (Theorem~\ref{lim2}) using
the concept of Hoffman graphs.

\begin{lemma}\label{lim1}
Let $\Go$ be a Hoffman graph whose vertex set is partitioned
as $V_1\cup V_2\cup V_3$ such that
\begin{enumerate}
\item $V_2\cup V_3\subset V_s(\Go)$,
\item there are no edges between $V_1$ and $V_3$,
\item every pair of vertices $x\in V_2$ and $y\in V_3$ are adjacent,
\item $V_3$ is a clique.
\end{enumerate}
Let $\Ho$ be the Hoffman graph with the set of vertices
$V_1\cup V_2$ together with a fat vertex $f\notin V(\Go)$ adjacent to all the
vertices of $V_2$. If $\Go$ has a representation of norm $m$,
then $\Ho$ has a representation of norm
\[
m+\frac{(m-1)|V_2|}{|V_3|+m-1}.
\]
\end{lemma}
\begin{proof}
Let $\phi: V(\Go)\to\R^d$ be a representation of norm $m$, and
let 
\[
P=\begin{pmatrix} P_1\\ P_2\\ P_3\end{pmatrix}
\]
be the $|V(\Go)|\times d$ matrix whose rows are the images of
$V(\Go)=V_1\cup V_2\cup V_3$ under $\phi$. Set
\begin{align*}
\bu&=\frac{1}{\sqrt{|V_3|(|V_3|+m-1)}}\sum_{x\in V_3}\phi(x),\\
\epsilon_1&=1-\sqrt{\frac{|V_3|}{|V_3|+m-1}},\\
\epsilon_2&=\sqrt{\frac{m-1}{|V_3|+m-1}}.
\end{align*}
Let $\bj$ denote the row vector of length $|V_2|$ all of whose
entries are $1$.
Then
\begin{align}
\bu\bu^T&=1,\label{e1}\\
P_1\bu^T&=0,\\
P_2\bu^T&=(1-\epsilon_1)\bj^T,\\
\epsilon_2^2&=2\epsilon_1-\epsilon_1^2.
\end{align}
Fix an orientation of the complete digraph on $V_2$, and let
$B$ be the $|V_2|\times\binom{|V_2|}{2}$ matrix defined by
\[
B_{\alpha,(\beta,\gamma)}=\delta_{\alpha\beta}-\delta_{\alpha\gamma}\quad
(\alpha,\beta,\gamma\in V_2,\;\beta\neq\gamma).
\]
Then
\begin{equation}\label{e5}
BB^T=|V_2|I-J.
\end{equation}
We now construct the desired representation of $\Ho$, as the row
vectors of the matrix
\[
D=
\begin{pmatrix}
P_1 &\epsilon_2\sqrt{|V_2|}I&0\\
P_2+\epsilon_1\bj^T\bu&0&\epsilon_2B\\
\bu&0&0
\end{pmatrix}.
\]
Then, using (\ref{e1})--(\ref{e5}), we find
\begin{align*}
DD^T&=
\begin{pmatrix}
P_1 &\epsilon_2\sqrt{|V_2|}I&0\\
P_2+\epsilon_1\bj^T\bu&0&\epsilon_2B\\
\bu&0&0
\end{pmatrix}
\begin{pmatrix}
P_1^T & P_2^T+\epsilon_1\bu^T\bj & \bu^T\\
\epsilon_2\sqrt{|V_2|}I&0&0\\
0&\epsilon_2B^T&0
\end{pmatrix}
\nexteq
\begin{pmatrix}
P_1P_1^T+\epsilon_2^2|V_2|I &
P_1P_2^T & 0\\
P_2P_1^T & P_2P_2^T+(2\epsilon_1-\epsilon_1^2)J+\epsilon_2^2(|V_2|I-J)&\bj^T\\
0&\bj&1
\end{pmatrix}
\nexteq
\begin{pmatrix}
P_1P_1^T+\epsilon_2^2|V_2|I &
P_1P_2^T & 0\\
P_2P_1^T & P_2P_2^T+\epsilon_2^2|V_2|I&\bj^T\\
0&\bj&1
\end{pmatrix}
\nexteq
\begin{pmatrix}
P_1P_1^T&
P_1P_2^T & 0\\
P_2P_1^T & P_2P_2^T&\bj^T\\
0&\bj&1
\end{pmatrix}
+\epsilon_2^2|V_2|
\begin{pmatrix}
I & 0&0\\
0&I&0\\
0&0&0
\end{pmatrix}
\end{align*}
Therefore, the row vectors of $D$ define a representation of
norm $m+\epsilon_2^2|V_2|$ of the Hoffman graph $\Ho$.
\end{proof}

\begin{theorem}[Hoffman]\label{lim2}
Let $\Ho$ be a Hoffman graph, and let $f_1,\dots,f_k\in V_f(\Ho)$. 
Let $\Go^{n_1,\dots,n_k}$
be the Hoffman graph obtained from $\Ho$ by replacing each $f_i$ by a
slim $n_i$-clique $K^i$, and joining all the neighbors of
$f_i$ with all the vertices of $K^i$ by edges.
Then
\begin{align}
\lambda_{\min}(\Go^{n_1,\dots,n_k})&\geq\lambda_{\min}(\Ho),
\label{eq:lim2ak}\\
\intertext{and}
\lim_{n_1,\dots,n_k\to\infty}\lambda_{\min}(\Go^{n_1,\dots,n_k})
&=\lambda_{\min}(\Ho).
\label{eq:lim2bk}
\end{align}
\end{theorem}
\begin{proof}
We prove the assertions by induction on $k$. First suppose $k=1$.
Let $\mu_n=-\lambda_{\min}(\Go^n)$. Let
$\Ho^n$ denote the Hoffman graph obtained 
from $\Ho$ by attaching a slim $n$-clique $K$ to the fat vertex $f_1$,
joining all the neighbors of $f_1$ and all the vertices of $K$
by edges. Then $\Ho^n$ contains both $\Ho$ and $\Go^n$ as subgraphs,
and $\Ho^n=\Ho\uplus\Ho'$, where $\Ho'$ is the subgraph induced
on $K\cup\{f_1\}$. Since $\lambda_{\min}(\Ho')=-1$, 
Lemma~\ref{lem:sumeigen} implies
\[
\lambda_{\min}(\Ho)=\lambda_{\min}(\Ho^n)\leq
\lambda_{\min}(\Go^n)=-\mu_n.
\]
Thus (\ref{eq:lim2ak}) holds for $k=1$.
Since $n$ is arbitrary and $\{-\mu_n\}_{n=1}^\infty$ is
decreasing, we see that $\lim_{n\to\infty}\mu_n$ exists and
\begin{equation}\label{lim2b}
\lambda_{\min}(\Ho)\leq-\lim_{n\to\infty}\mu_n.
\end{equation}

Since
$\Go^n$ has a representation of norm $\mu_n$,
it follows from Lemma~\ref{lim1} that $\Ho$ has a representation
of norm 
\[
\mu_n+\frac{(\mu_n-1)|N_{\Ho}(f)|}{n+\mu_n-1}.
\]
By Theorem~\ref{thm:WN1},
	we have
\[
\lambda_{\min}(\Ho)
\geq
-\mu_n-\frac{(\mu_n-1)|N_{\Ho}(f)|}{n+\mu_n-1},
\]
which implies 
\begin{equation}\label{lim2a}
\lambda_{\min}(\Ho)\geq-\lim_{n\to\infty}\mu_n.
\end{equation}
Combining (\ref{lim2a}) with (\ref{lim2b}), we conclude that
(\ref{eq:lim2bk}) holds for $k=1$. 

Next, suppose $k\geq2$. 
Let $\Go^{n_1,\dots,n_{k-1}}$
be the Hoffman graph obtained from $\Ho$ by replacing each $f_i$ 
($1\leq i\leq k-1$) by a
slim $n_i$-clique $K^i$, and joining all the neighbors of
$f_i$ with all the vertices of $K^i$ by edges.
Then $\Go^{n_1,\dots,n_{k}}$ is obtained from $\Go^{n_1,\dots,n_{k-1}}$
by replacing $f_k$ by a slim $n_k$-clique $K^k$, 
and joining all the neighbors of
$f_k$ with all the vertices of $K^k$ by edges.
Then it follows from the assertions for $k=1$ that
\begin{align}
\lambda_{\min}(\Go^{n_1,\dots,n_k})&\geq\lambda_{\min}
(\Go^{n_1,\dots,n_{k-1}}),
\label{eq:lim2ak1}\\
\intertext{and}
\lim_{n_k\to\infty}\lambda_{\min}(\Go^{n_1,\dots,n_k})
&=\lambda_{\min}(\Go^{n_1,\dots,n_{k-1}}).
\notag
\end{align}
This means that, for any $\epsilon>0$, there exists
$N_1$ such that
\[
n_k\geq N_1\implies 0\leq\lambda_{\min}(\Go^{n_1,\dots,n_k})
-\lambda_{\min}(\Go^{n_1,\dots,n_{k-1}})<\epsilon.
\]
By induction, we have
\begin{align}
\lambda_{\min}(\Go^{n_1,\dots,n_{k-1}})&\geq
\lambda_{\min}(\Ho),\label{eq:I1}\\
\intertext{and}
\lim_{n_1,\dots,n_{k-1}\to\infty}\lambda_{\min}
(\Go^{n_1,\dots,n_{k-1}})&=\lambda_{\min}(\Ho).
\label{eq:I2}
\end{align}
Combining (\ref{eq:lim2ak1}) with (\ref{eq:I1}), we obtain
(\ref{eq:lim2ak}), while (\ref{eq:I1}) and (\ref{eq:I2})
imply that there exists $N_0$ such that
\[
n_1,\dots,n_{k-1}\geq N_0\implies 
0\leq\lambda_{\min}(\Go^{n_1,\dots,n_{k-1}})
-\lambda_{\min}(\Ho)<\epsilon.
\]
Setting $N=\max\{N_0,N_1\}$, we see that
\[
n_1,\dots,n_{k}\geq N\implies 
0\leq\lambda_{\min}(\Go^{n_1,\dots,n_{k}})
-\lambda_{\min}(\Ho)<2\epsilon.
\]
This establishes (\ref{eq:lim2bk}).
\end{proof}


\begin{cor}\label{lim3}
Let $\Ho$ be a Hoffman graph. Let $\Gamma^n$
be the slim graph obtained from $\Ho$ by replacing 
every fat vertex $f$ of $\Ho$ by a
slim $n$-clique $K(f)$, and joining all the neighbors of
$f$ with all the vertices of $K(f)$ by edges.
Then 
\begin{align*}
\lambda_{\min}(\Gamma^n)&\geq\lambda_{\min}(\Ho),
\\
\intertext{and}
\lim_{n\to\infty}\lambda_{\min}(\Gamma^n)&=\lambda_{\min}(\Ho).
\end{align*}
In particular,
for any $\epsilon >0$,
there exists a natural number $n$ such that,
every slim graph $\Delta$ containing $\Gamma^n$ as an induced
subgraph satisfies
\[
\lambda_{\min}(\Delta)\leq\lambda_{\min}(\Ho)+\epsilon.
\]
\end{cor}
\begin{proof}
Immediate from Theorem~\ref{lim2}.
\end{proof}

\section{Special graphs of Hoffman graphs}


\begin{defi}\label{def:max}
Let $\mu$ be a real number with $\mu\leq-1$
and let $\Ho$ be a Hoffman graph with smallest eigenvalue at least $\mu$.
Then $\Ho$ is called 
\emph{\sat{\mu}}
if no fat vertex can be attached to
$\Ho$ in such a way that the resulting graph has 
smallest eigenvalue at least $\mu$.
%
\end{defi}


\begin{prop}\label{prop:line}
Let $\mu$ be a real number, and let $\h$ be a family of 
indecomposable fat Hoffman graphs with smallest eigenvalue
at least $\mu$.
The following statements are equivalent:
\begin{enumerate}
\item
every fat Hoffman graph with smallest eigenvalue at least $\mu$ 
is a subgraph of a graph $\Ho = 
\biguplus_{i=1}^n \Ho^i$ such that $\Ho^i$ is a member of $\h$
for all $i=1,\dots,n$.
\item
every \sat{\mu}\ indecomposable fat Hoffman graph is isomorphic to
a subgraph of a member of $\h$.
\end{enumerate}
\end{prop}
\begin{proof}
First suppose (i) holds, and let $\Ho$ be a
\sat{\mu}\ indecomposable fat Hoffman 
graph. Then $\Ho$ is a fat Hoffman graph with smallest eigenvalue
at least $\mu$, hence $\Ho$ is 
a subgraph of $\Ho'=\biguplus_{i=1}^n \Ho^i$, where $\Ho^i$ is
a member of $\h$ for $i=1,\dots,n$. Since $\Ho$ is 
\sat{\mu}, it coincides with
the subgraph $\subgg{V_s(\Ho)}{\Ho'}$ of $\Ho'$. Since $\Ho$ is 
indecomposable, this implies that $\Ho$ is a subgraph of $\Ho^i$ for some
$i$. 

Next suppose (ii) holds, and let $\Ho$ be a fat Hoffman graph
with smallest eigenvalue at least $\mu$. Without loss of generality
we may assume that $\Ho$ is indecomposable and \sat{\mu}. 
Then $\Ho$ is isomorphic to a subgraph of a member of $\h$, hence 
(i) holds.
\end{proof}

\begin{defi}\label{df:special}
For a Hoffman graph $\Ho$,
	we define the following three graphs $\TS^-(\Ho)$,
	$\TS^+(\Ho)$ and $\TS(\Ho)$ as follows:
For $\epsilon \in \{-, +\}$ define
	the {\em special $\epsilon$-graph} $\TS^{\epsilon}=(V_s(\Ho), E^{\epsilon})$
	as follows:
the set of edges $E^{\epsilon}$ consists of pairs $\{s_1,s_2\}$ of distinct
slim vertices such that $\sgn(\psi(s_1), \psi(s_2)) = \epsilon$,
	where $\psi$ is a reduced representation of $\Ho$ of norm $m$.
The graph $\TS(\Ho):= \TS^+(\Ho) \cup \TS^-(\Ho)
	= (V_s(\Ho), E^{-} \cup E^{+})$
	is the {\em special graph} of $\Ho$.
\end{defi}

Note that the definition of the special graph $\TS(\Ho)$ is independent of
the choice of the norm $m$ of a reduced representation $\psi$.

It is easy to determine whether a Hoffman graph $\Ho$
is decomposable or not.


\begin{lemma}\label{lm:s2}
Let $\Ho$ be a Hoffman graph.
Let $V_s(\Ho)=V_1\cup V_2$ be a partition,
and set $\Ho^i=\subgg{V_i}{\Ho}$ for $i=1,2$.
Then $\Ho=\Ho^1\uplus \Ho^2$ if and only if
there are no edges connecting $V_1$ and $V_2$ in $\TS(\Ho)$.
In particular, $\Ho$ is indecomposable if and only if
$\TS(\Ho)$ is connected.
\end{lemma}

\begin{proof}
This is immediate from Definition~\ref{df:sum}(iv) and Definition~\ref{df:special}.
\end{proof}


For an integer $t \geq 1$,
	let $\Ho^{(t)}$ be the fat Hoffman graph
	with one slim vertex and $t$ fat vertices. 

\begin{figure}[h]
\caption{}
\begin{center}
\begin{tabular}{ccccc}
\includegraphics[scale=0.15]{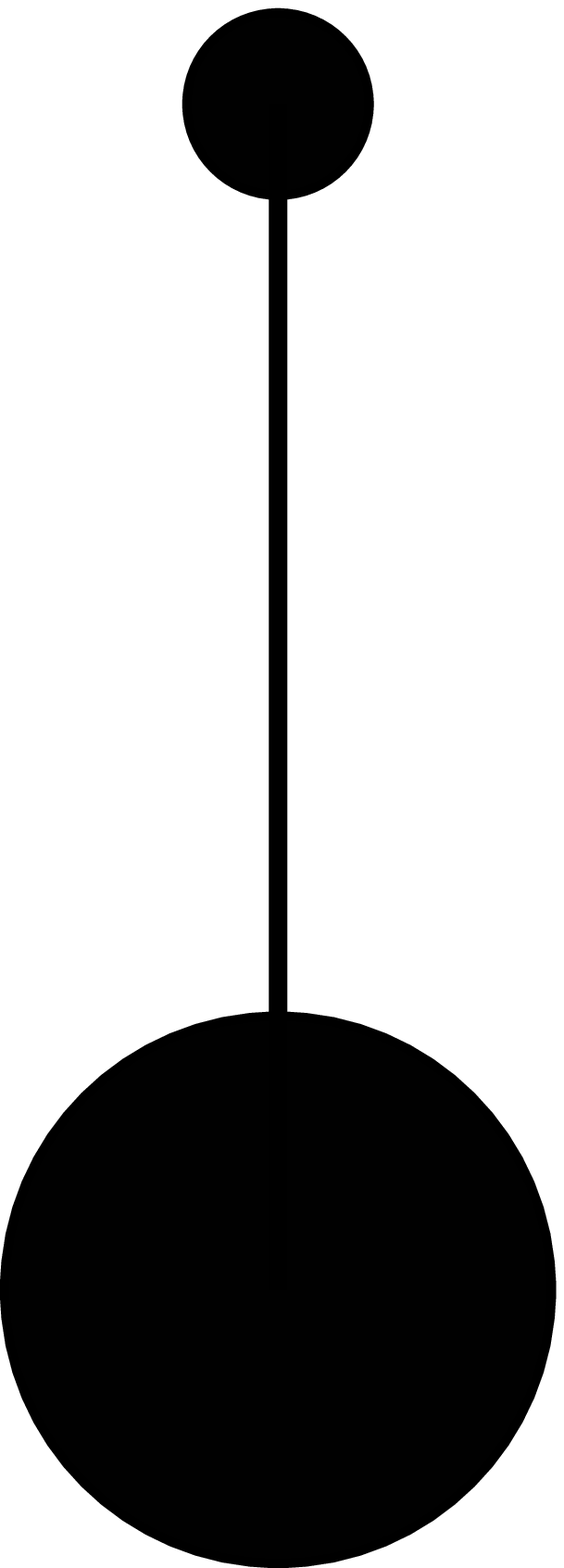} &&
\includegraphics[scale=0.15]{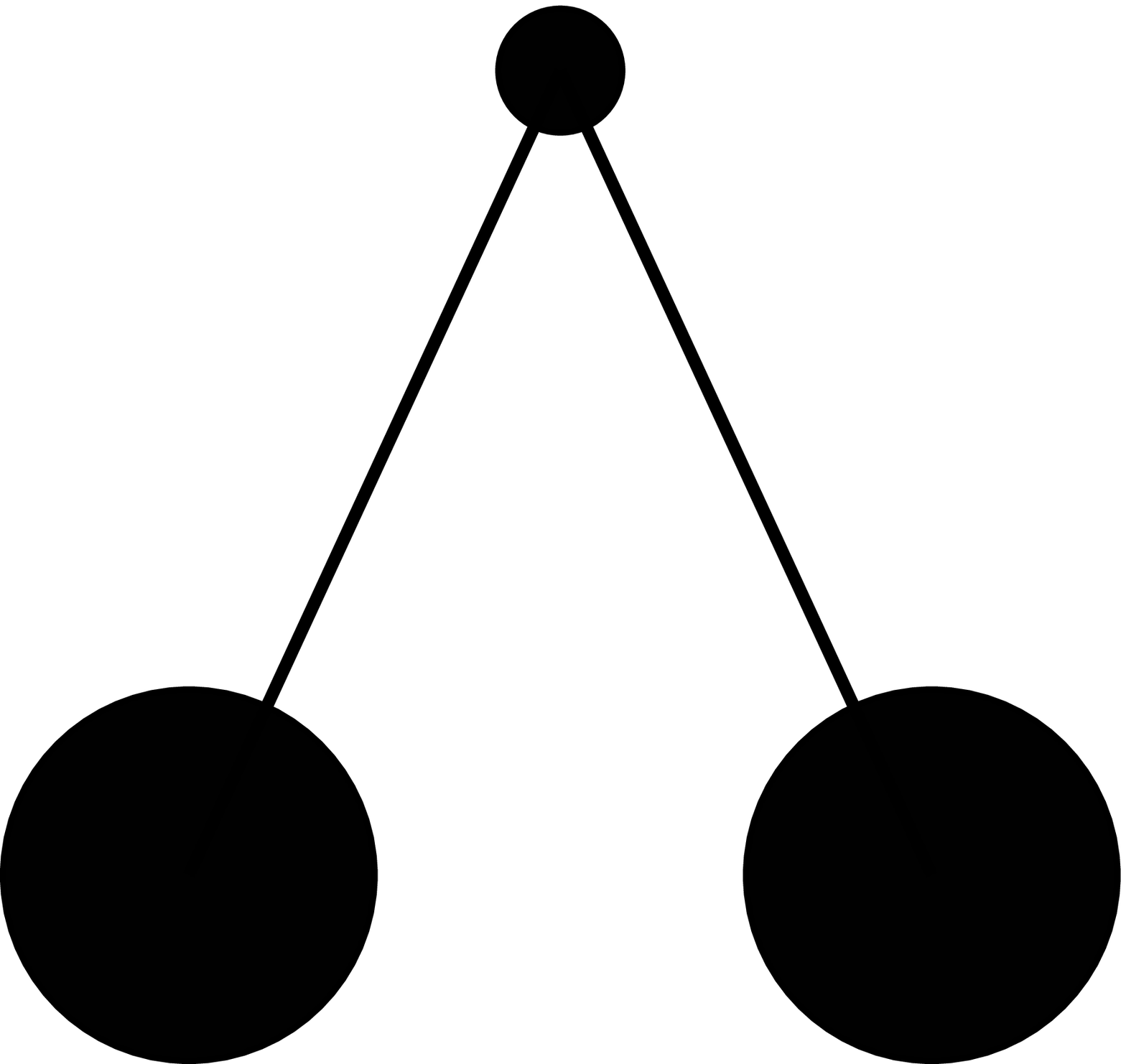} &&
\includegraphics[scale=0.15]{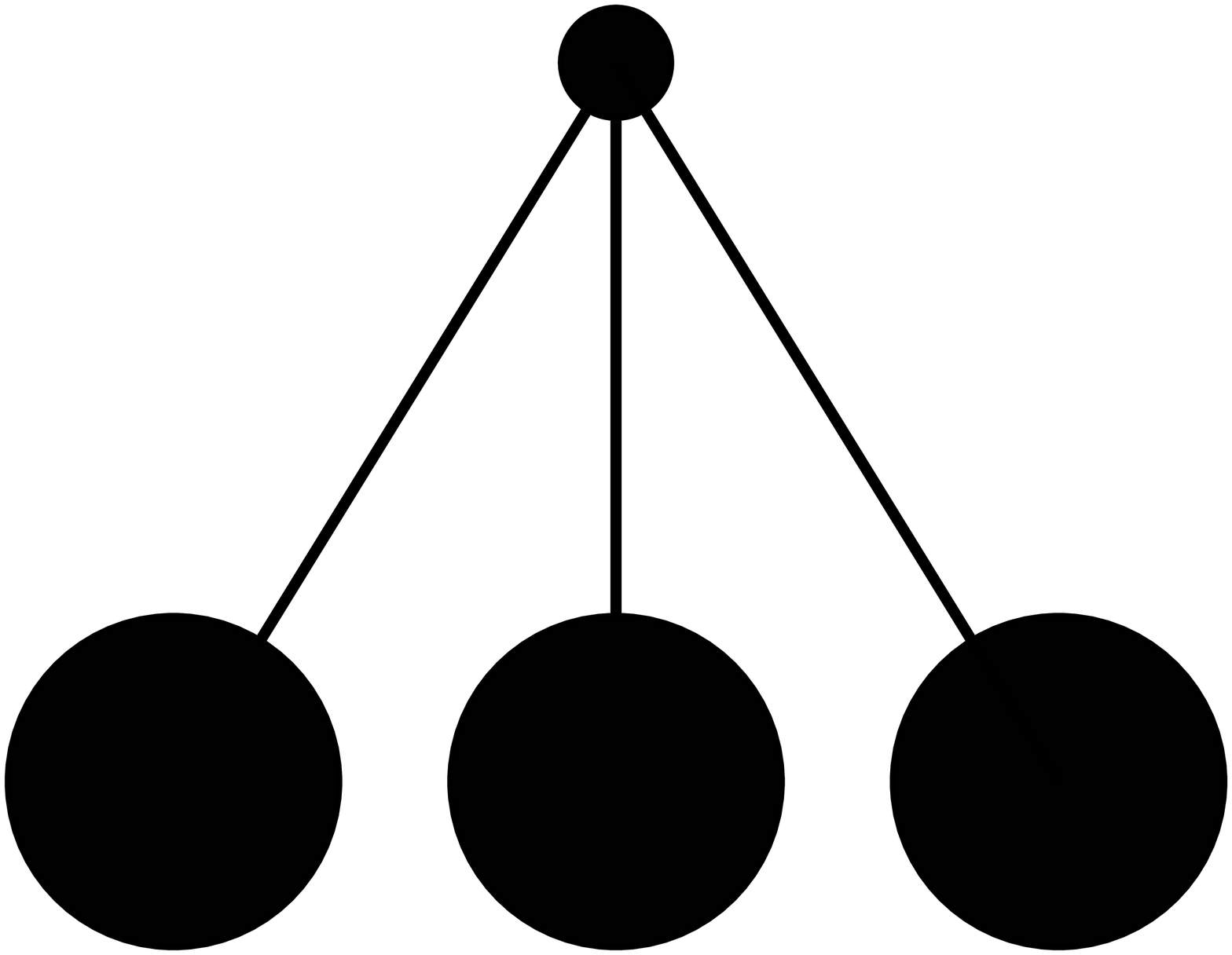}\\
$\Ho^{(1)},\lambda_{\min}=-1$ &&
$\Ho^{(2)},\lambda_{\min}=-2$ &&
$\Ho^{(3)},\lambda_{\min}=-3$\\
\end{tabular}
\end{center}
\label{Hoffmans0}
\end{figure}

\begin{lemma}\label{lem:sat}
Let $t$ be a positive integer.
If $\Ho$ is a fat Hoffman graph
with $\lambda_{\min}(\Ho)\geq-t$
containing $\Ho^{(t)}$ as a Hoffman subgraph,
then $\Ho=\Ho^{(t)}\uplus\Ho'$ for some subgraph $\Ho'$ of $\Ho$. In particular, if $\Ho$ is indecomposable, then $\Ho=\Ho^{(t)}$.
\end{lemma}

\begin{proof}
Let  $x$ be the unique slim vertex of $ \Ho^{(t)} $.
Let $\psi$ be a reduced representation of norm $t$ of $\Ho$.
Then $\psi(x) = 0$,
	hence $x$ is an isolated vertex in $\TS(\Ho)$.
Thus
$\Ho=\Ho^{(t)}\uplus\subgg{V_s(\Ho)\setminus \{x\}}{\Ho}$
by Lemma~\ref{lm:s2}.
\end{proof} 


\begin{lemma}\label{lm:007x}
Let $\Ho$ be a fat Hoffman graph with smallest eigenvalue at least $-3$. 
Let $\psi$ be a reduced representation of norm $3$ of $\Ho$. 
Then for any distinct slim vertices $x,y$ of $\Ho$,
$(\psi(x),\psi(y))\in\{1,0,-1\}$.
\end{lemma}
\begin{proof}
Since $\Ho$ is fat, we have $(\psi(x),\psi(x))\leq2$ for all $x\in V_s(\Ho)$.
Thus $|(\psi(x),\psi(y))|\leq2$ for all $x,y\in V_s(\Ho)$ by 
Schwarz's inequality. Equality holds only if
$\psi(x)=\pm\psi(y)$ and $(\psi(x),\psi(x))=2$. The latter condition
implies $|N_{\Ho}^f(x)|=1$, hence $|N_{\Ho}^f(x,y)|\leq1$. Thus
$(\psi(x),\psi(y))\geq-1$, while $(\psi(x),\psi(y))=2$ cannot occur
unless $x=y$, by Definition~\ref{df:redrep}. Therefore, 
$|(\psi(x),\psi(y))|<2$, and the result follows.
\end{proof}

Let $\Ho$ be a fat Hoffman graph with smallest eigenvalue at least $-3$. 
Then by Lemma~\ref{lm:007x}, the edge set 
of the special graph $\TS^{\epsilon}(\Ho)$ is 
\[
\{\{x,y\}\mid x,y\in V_s(\Ho),\;(\psi(x),\psi(y))=\epsilon1\},
\]
for $\epsilon\in\{+,-\}$.

\begin{theorem}\label{thm:007}
Let $\Ho$ be a fat 
indecomposable Hoffman graph
	with smallest eigenvalue at least $-3$.
Then every slim vertex has at most three fat neighbors.
Moreover, the following statements hold:
\begin{enumerate}[{\rm(i)}]
\item\label{007-0} If some slim vertex has three fat neighbors,
	then $\Ho\cong\Ho^{(3)}$.
\item\label{007-1}
	 If no slim vertex has three fat neighbors and some slim vertex has exactly two fat neighbors,
	 then $\redlam (\Ho,3) \simeq \zz^n$ for 
	some positive integer $n$.
\item\label{007-2} 
If every slim vertex has a unique fat neighbor,
	then $\redlam (\Ho,3)$ is an irreducible root lattice.
\end{enumerate}
\end{theorem}

\begin{proof}
As the smallest eigenvalue is at least $-3$,
	every slim vertex has at most three fat neighbors.

If $|N_{\Ho}^f(x)|=3$ for some slim vertex $x$ of $\Ho$,
	then $\Ho$ contains $\Ho^{(3)}$ as a subgraph.
Thus $\Ho=\Ho^{(3)}$ by Lemma~\ref{lem:sat},
	and (i) holds.
Hence we may assume 
that  $|N_{\Ho}^f(x)|\le 2$ for all $x\in V_s(\Ho)$.
Then for each $x\in V_s(\Ho)$ we have $\|\psi (x)\|^2=1$ (resp.\  $2$)
	if and only if $|N_{\Ho}^f(x)|=2 $ (resp.\ $|N_{\Ho}^f(x)|=1$).
Suppose that $\redlam (\Ho,3)$ contains $m$ linearly independent
vectors of norm $1$. We claim that
$\redlam (\Ho,3)$ can be written as an orthogonal direct sum
$\zz^m\oplus\Lambda'$, where $\Lambda'$ is a lattice
containing no vectors of norm $1$. Indeed, if
$x$ is a slim vertex such that $\|\psi(x)\|^2=2$
and $\psi(x)\notin\zz^m$, then $\psi(x)$ is orthogonal to
$\zz^m$. This implies
$\redlam (\Ho,3)=\zz^m\oplus\Lambda'$ and
$\psi(V_s(\Ho))\subset\zz^m\cup\Lambda'$.

If $m>0$ and $\Lambda'\neq0$, then the special graph
$\TS(\Ho)$ is disconnected. This contradicts the 
indecomposability of $\Ho$ by Lemma~\ref{lm:s2}.
Therefore, we have either $m=0$ or $\Lambda'=0$.
In the latter case, (ii) holds. In the former case,
$\redlam (\Ho,3)=\Lambda'$ is generated by vectors of
norm $2$, hence it is a root lattice. Again by
the assumption and Lemma~\ref{lm:s2}, 
(iii) holds.
\end{proof}

We shall see some examples for the case (ii) of Theorem~\ref{thm:007}
in the next section. As for the case (iii), 
$\redlam(\Ho,3)$ is an irreducible root lattice of
type $\A_n$, $\D_n$ or $\E_n$.
If $\redlam(\Ho,3)$ is not an irreducible root lattice of
type $\E_n$, then it can be imbedded into the standard lattice,
hence the results of the next section applies. 
On the other hand,
if $\redlam(\Ho,3)$ is an irreducible root lattice of
type $\E_n$, then it is contained in the irreducible
root lattice of type $\E_8$, and hence there are only
finitely many possibilities. For example,
Let $\Gamma$ be
any ordinary graph with smallest eigenvalue at least $-2$
(see \cite{new} for a description of such graphs).
Attaching a fat neighbor to each vertex of $\Gamma$ gives a fat
Hoffman graph with smallest eigenvalue at least $-3$.
However, this Hoffman graph may not be maximal among
fat Hoffman graphs with smallest eigenvalue at least $-3$.
Therefore, we aim to classify fat Hoffman graphs with
smallest eigenvalue at least $-3$ which are maximal in
$\E_8$. This may seem a computer enumeration problem,
but it is harder than it looks.

\begin{exa}\label{ex:mE8}
Let $\Pi$ denote the root system of type $\E_8$. Fix
$\alpha\in\Pi$. Then there are elements $\beta_i\in\Pi$
($i=1,\dots,28$) such that
\[
\{\beta\in\Pi\mid(\alpha,\beta)=1\}=
\bigcup_{i=1}^{28}\{\beta_i,\alpha-\beta_i\}.
\]
Let $V$ denote the set of $57$ roots consisting of the above set
and $\alpha$. Then $V$
is a reduced representation of 
a fat Hoffman graph $\Ho$ with $29$ fat
vertices.
The fat vertices of $\Ho$ are $f_i$ ($i=0,1,\dots,28$),
$f_0$ is adjacent to $\alpha$, 
and $f_i$ is adjacent to $\beta_i,\alpha-\beta_i$
($i=1,\dots,28$).
It turns out that $\Ho$ is maximal among 
fat Hoffman graphs with smallest eigenvalue at least $-3$.
Indeed, no fat vertex can be attached, since the root lattice
of type $\E_8$ is generated by $V\setminus\{\gamma\}$ for
any $\gamma\in V$, and attaching another fat neighbor to
$\gamma$ would mean the existence of a vector of norm $1$
in the dual lattice $\E_8^*$ of $\E_8$. 
Since $\E_8^*=\E_8$, there are no vectors of norm $1$ in $\E_8^*$.
This is a contradiction. If a slim vertex can be attached,
then it can be represented by $\delta\in\Pi$ with $(\alpha,\delta)=0$.
Then there exists $i\in\{1,\dots,28\}$ such that
$(\beta_i,\delta)=\pm1$. Exchanging $\beta_i$ with $\alpha-\beta_i$
if necessary, we may assume $(\beta_i,\delta)=-1$.
This implies that $\beta_i$ and $\delta$
have a common fat neighbor. Since
$(\beta_i,\alpha-\beta_i)=-1$, $\beta_i$ and $\alpha-\beta_i$
have a common fat neighbor. Since $\beta_i$ has a unique fat
neighbor, $\delta$ and $\alpha-\beta_i$
have a common fat neighbor, contradicting
$(\delta,\alpha-\beta_i)=1$.
\end{exa}

On the other hand, it is known that there is a slim graph $\Gamma$
with $36$ vertices represented by the root system of type $\E_8$
(see \cite{new}). Attaching a fat neighbor to each vertex of
$\Gamma$ gives a fat Hoffman graph $\Ho'$ with smallest eigenvalue $-3$
such that $\redlam(\Ho',3)$ is isometric to the root lattice of
type $\E_8$. The graph $\Ho'$ is not contained in $\Ho$,
so there seems a large number of maximal fat Hoffman graphs
represented in the root lattice of type $\E_8$.

\section{Integrally represented Hoffman graphs}\label{sec:int}


In this section, we consider a fat \sat{(-3)}\ Hoffman graph
$\Ho$ such that $\redlam(\Ho,3)$ is a sublattice of the standard
lattice $\zz^n$. 
Since any of the exceptional root lattices $\E_6,\E_7$ and $\E_8$
cannot be embedded into the standard lattice (see \cite{Eb}),
this means that, in view of Theorem~\ref{thm:007},
$\redlam(\Ho,3)$ is isometric to $\zz^n$ or a root lattice of
type $\A_n$ or $\D_n$. Note that $\redlam(\Ho,3)$ cannot be 
isometric to the lattice $\A_1$, since this would imply that $\Ho$
has a unique slim vertex with a unique fat neighbor, contradicting
\sat{(-3)}ness. The following example gives a
fat \sat{(-3)}\ graph $\Ho$ with $\redlam(\Ho,3)\cong\zz$.

\begin{exa}\label{exa:A3t}
Let $\Ho$ be the Hoffman graph with vertex set $V_s(\Ho)\cup V_f(\Ho)$,
where $V_s(\Ho)=\zz/4\zz$, $V_f(\Ho)=\{f_i\mid i\in\zz/4\zz\}$, and with
edge set
\[
\{\{0,2\},\{1,3\}\}\cup\{\{i,f_j\}\mid i=j\text{ or }j+1\}.
\]
Then $\Ho$ is a fat indecomposable \sat{(-3)}\ Hoffman graph
such that 
$\redlam (\Ho,3)$ is isomorphic to the standard lattice $\zz$. Since
$\TS^-(\Ho)$ has edge set $\{\{i,i+1\}\mid i\in\zz/4\zz\}$, 
$\TS^-(\Ho)$ is isomorphic to the Dynkin graph $\tilde{\A}_3$. 
Theorem~\ref{thm:001} below implies that $\Ho$ is maximal, in the
sense that no fat indecomposable \sat{(-3)}\ Hoffman graph
contains $\Ho$.
\end{exa}

For the remainder of this section, 
we let $\Ho$ be a fat indecomposable \sat{(-3)}\
Hoffman graph 
such that 
$\redlam(\Ho,3)$ is isomorphic to a sublattice of the standard lattice
$\zz^n$.
Let $\phi$
be a representation of norm $3$ of $\Ho$. 
Then we may assume that
$\phi$ is a mapping from $V(\Ho)$ to $\zz^n\oplus\zz^{V_f(\Ho)}$,
where its composition with the projection $\zz^n\oplus\zz^{V_f(\Ho)}
\to\zz^n$ gives a reduced representation
$\psi:V_s(\Ho)\to\zz^n$. It follows from the definition of
a representation of norm $3$ that
\[
\phi(s)=\psi(s)+\sum_{f\in N_{\Ho}^f(s)} \be_f,
\]
\[
\psi(s)=\sum_{j=1}^n\psi(s)_j\be_j,
\]
\[
\psi(s)_j\in\{0,\pm1\},
\]
and
\begin{equation}\label{eq:psis}
|\{j\mid j\in\{1,\dots,n\},\;\psi(s)_j\in\{\pm1\}\}|
=3-|N_{\Ho}^f(s)|\leq2.
\end{equation}

\begin{lemma}\label{lm:002}
If $i\in\{1,\dots,n\}$ and $\psi(s)_i\neq0$ for some $s\in V_s(\Ho)$,
then there exist $s_1,s_2\in V_s(\Ho)$
such that $\psi(s_1)_i=-\psi(s_2)_i=1$.
\end{lemma}
\begin{proof}
By way of contradiction,
we may assume without loss of generality that $i=n$, and
$\psi(s)_n\in\{0,1\}$ for all $s\in V_s(\Ho)$. 
Let $\Go$ be the Hoffman graph obtained from $\Ho$ by attaching
a new fat vertex $g$ and join it to all the slim vertices $s$ of $\Ho$
satisfying $\psi(s)_n=1$. Then the composition of 
$\psi:V_s(\Ho)=V_s(\Go)\to\zz^n$ and the projection
$\zz^n\to\zz^{n-1}$ gives a reduced representation of norm $3$
of $\Go$. By Theorem~\ref{thm:WN1}, $\Go$ has smallest eigenvalue
at least $-3$. This contradicts the assumption that $\Ho$ is
\sat{(-3)}.
\end{proof}

\begin{prop}\label{p:001}
The graph 
$\TS^-(\Ho)$ is connected. 
\end{prop}
\begin{proof}
Before proving the proposition, we first show the following claim.
\begin{claim}\label{cl:001}
Let $s _1$ and $s_2$ be two slim vertices such that
	$\phi(s_1)_i = 1$ and $\phi(s_2)_i= -1$ for some 
	$i \in \{1,\dots,n\}$.
Then the distance between $s_1$ and $s_2$ is at most $2$ in $\TS^-(\Ho)$.
\end{claim}
By (\ref{eq:psis}), we have $(\psi(s_1), \psi(s_2))\in\{0,-1\}$.
If $(\psi(s_1), \psi(s_2)) = -1$, then 
$s_1$ and $s_2$ are adjacent in $\TS^-(\Ho)$ by the definition, hence
the distance equals one. 
If $(\psi(s_1), \psi(s_2)) = 0$, then 
there exists a unique $j \in \{1,\dots,n\}$ such that
	$\phi(s_1)_j =\phi(s_2)_j=\pm1$.
From Lemma~\ref{lm:002},
there exists a slim vertex $s_3$ such that 
$\phi(s_3)_j = -\phi(s_1)_j$.
If $\phi (s_3)_i\neq 0$,
	then $(\psi (s_q),\psi (s_3))\in\{\pm2\}$
	for some $q\in\{1,2\}$, which is a contradiction.
Hence $\phi(s_3)_i = 0$.
This implies that $(\psi(s_i),\psi(s_3))=-1$ for $i=1,2$, or equivalently,
$s_3$ is a common neighbor of $s_1$ and $s_2$ in $\TS^-(\Ho)$.
This shows the claim.

Since $\Ho$ is indecomposable, $\TS(\Ho)$ is connected by Lemma~\ref{lm:s2}.
Thus, in order to show the proposition, we only need to show that
	slim vertices $s_1$ and $s_2$ with 
	$(\psi(s_1), \psi(s_2)) = 1$ are connected by a path in $\TS^-(\Ho)$. 
There exists $i\in\{1,\dots,n\}$ 
such that $\phi(s_1)_i = \phi(s_2)_i=\pm1$.
From Lemma~\ref{lm:002},
there exists a slim vertex $s_3$ such that $\phi(s_3)_i = -\phi(s_1)_i$,
and hence the distances between $s_1$ and $s_3$ and between $s_3$ and $s_2$ 
are at most $2$ in $\TS^-(\Ho)$ by Claim~\ref{cl:001}.
Therefore, $s_1$ and $s_2$ are connected by a path of length at most $4$
in $\TS^-(\Ho)$.
\end{proof}

\begin{lemma}\label{lm:inj}
Let $\Ho$ be a fat indecomposable \sat{(-3)}\ Hoffman graph.
Then the reduced representation of norm $3$ of $\Ho$ is injective
unless $\Ho$ is isomorphic to a subgraph of the graph given in
Example~\ref{exa:A3t}.
\end{lemma}
\begin{proof}
Suppose that the reduced representation $\psi$ of norm $3$ of
$\Ho$ is not injective. Then there are two distinct
slim vertices $x$ and $y$ satisfing $\psi(x)=\psi(y)$.
Then
$(\psi(x),\psi(y))=0$ or $1$. 

If $(\psi(x),\psi(y))=0$, then
$\psi(x)=\psi(y)=0$, hence both $x$ and $y$ are isolated vertices,
contradicting the assumption that $\Ho$ is indecomposable.

Suppose $(\psi(x),\psi(y))=1$.
Since $\TS^-(\Ho)$ is connected by Proposition~\ref{p:001},
there exists a slim vertex $z$ such that $(\psi(x),\psi(z))=-1$.
Then we may assume
\begin{align*}
\phi(x)&=\be_1+\be_{f_1}+\be_{f_2},\\
\phi(y)&=\be_1+\be_{f_3}+\be_{f_4},\\
\phi(z)&=-\be_1+\be_{f_1}+\be_{f_3}.\\
\intertext{If $\Ho$ has another slim vertex $w$, then}
\phi(w)&=-\be_1+\be_{f_2}+\be_{f_4},
\end{align*}
and no other possibility occurs. Therefore, $\Ho$ is
isomorphic to either the graph given in Example~\ref{exa:A3t},
or its subgraph obtained by deleting one slim vertex.
\end{proof}

\begin{lemma}\label{lm:003_1}
Suppose that $s\in V_s(\Ho)$ has exactly two fat neighbors in $\Ho$.
Then the following statements hold.
\begin{enumerate}
\item for each $f\in N_{\Ho}^f(s)$,
$|N_{{\TS}^-(\Ho)}(s)\cap N_{\Ho}^s(f)|\le 2$,
\item
$|N_{{\TS}^-(\Ho)}(s)|\leq4$, and if equality holds, then
$\TS^-(\Ho)$ is isomorphic to the
graph $\tilde{\D}_4$,
\item
if $|N_{{\TS}^-(\Ho)}(s)|=3$, then 
two of the vertices in $N_{{\TS}^-(\Ho)}(s)$ have $s$ as their
unique neighbor in $\TS^-(\Ho)$.
\end{enumerate}
\end{lemma}
\begin{proof}
Let $\psi$ be the reduced representation of norm $3$ of $\Ho$.
Since $s$ has exactly two fat neighbors, 
$(\psi(s),\psi(s))=1$. This means that we may assume without loss
of generality 
$\psi(s)=\be_1$.

Let $f\in N_{\Ho}^f(s)$.
If $t_1$ and $t_2$ are distinct vertices of 
$N_{{\TS}^-(\Ho)}(s)\cap N_{\Ho}^s(f)$, 
then
\begin{align*}
1&\geq(\phi(t_1),\phi(t_2))
\\&\geq
(\psi(t_1)+\be_f,\psi(t_2)+\be_f)
\nexteq
(\psi(t_1),\psi(t_2))+1,
\end{align*}
Thus we have $(\psi(t_1),\psi(t_2))\leq0$. 
Since $t_1,t_2\in N_{\TS^-(\Ho)}(s)$,
we have
$(\psi(s),\psi(t_1))$ $=(\psi(s),\psi(t_2))=-1$, and hence
we may assume without loss of generality that
\begin{align}
\psi(t_1)&=-\be_1+\be_2,\label{t1}\\
\psi(t_2)&=-\be_1-\be_2.\label{t2}
\end{align}
Now it is clear that there cannot be another $t_3\in N_{\TS^-(\Ho)}(s)$.
Thus $|N_{{\TS}^-(\Ho)}(s)\cap N_{\Ho}^s(f)|\le 2$.
This proves (i).

As for (ii), let $N_{\Ho}^f(s)=\{f,f'\}$. Then
\[|N_{{\TS}^-(\Ho)}(s)|\leq|N_{{\TS}^-(\Ho)}(s)\cap  N_{\Ho}^s(f)|+
|N_{{\TS}^-(\Ho)}(s)\cap  N_{\Ho}^s(f')|\leq4
\]
by (i).
To prove (iii) and the second part of (ii), 
we may assume that
$\{t_1,t_2\}=N_{\TS^-(\Ho)}(s)\cap N_{\Ho}(f)$.
We claim that neither $t_1$ nor $t_2$ has a neighbor in $\TS^-(\Ho)$
other than $s$. Suppose by contradiction, that $t_3\neq s$ is a neighbor
in $\TS^-(\Ho)$ of $t_1$. By (\ref{t1}) (resp.\ (\ref{t2})), 
$f$ is the unique fat
neighbor of $t_1$ (resp.\ $t_2$). In particular, $f$ is also a neighbor of $t_3$.
Observe
\begin{align*}
1&\geq(\phi(s),\phi(t_3))\geq(\psi(s),\psi(t_3))+1,\\
1&\geq(\phi(t_i),\phi(t_3))=(\psi(t_i),\psi(t_3))+1\quad(i=1,2).
\end{align*}
Thus
\begin{align*}
(\be_1,\psi(t_3))&\leq0,\\
(-\be_1\pm\be_2,\psi(t_3))&\leq0.
\end{align*}
These imply $(\be_1,\psi(t_3))=(\be_2,\psi(t_3))=0$. But then
$-1=(\psi(t_1),\psi(t_3))=(-\be_1+\be_2,\psi(t_3))=0$.
This is a contradiction, and we have proved our claim.
Now (iii) is an immediate consequence of this claim.

Continuing the proof of the second part of (ii),
if $|N_{{\TS}^-(\Ho)}(s)|=4$, then we may assume
\begin{align*}
\psi(t'_1)&=-\be_1+\be_3,\\
\psi(t'_2)&=-\be_1-\be_3,
\end{align*}
where 
$\{t'_1,t'_2\}= N_{{\TS}^-(\Ho)}(s)\cap  N_{\Ho}^s(f')$.
It follows that $t_1,t_2,t'_1,t'_2$ are pairwise non-adjacent in 
$\TS^-(\Ho)$. By our claim, none of $t_1,t_2,t'_1,t'_2$ is
adjacent to any vertex other than $s$ in $\TS^-(\Ho)$.
Since $\TS^-(\Ho)$ is connected by Proposition~\ref{p:001},
$\TS^-(\Ho)$ is isomorphic to the graph $\tilde{\D}_4$.
\end{proof}

\begin{lemma}\label{lm:004_1}
Suppose that slim vertices $s,t^+,t^-$ share a common fat neighbor
and that they are represented as
\begin{align*}
\psi(s)&=\be_1+\be_2,\\
\psi(t^\pm)&=-\be_1\pm\be_3.
\end{align*}
If there exists a slim vertex $t$ with 
\[
\psi(t)=-\be_2+\be_j\quad\text{for some $j\notin\{1,2,3\}$,}
\]
then the vertices $t^\pm$ have $s$ as their unique neighbor in
$\TS^-(\Ho)$.
\end{lemma}
\begin{proof}
Note that each of the vertices $s,t^\pm,t$ has a unique fat
neighbor. Since $(\psi(s),\psi(t^\pm))=(\psi(s),\psi(t))=-1$,
these vertices share a common fat neighbor $f$.
Suppose that there exists a slim vertex $t'$ adjacent to
$t^-$ in $\TS^-(\Ho)$. This means $(\psi(t'),\psi(t^-))=-1$.
Since $f$ is the unique fat neighbor of $t^-$, 
$t'$ is adjacent to $f$, and
hence $(\psi(t'),\psi(u))\leq0$ for $u\in\{t,t^+,s\}$.
This is impossible.
Similarly, there exists no slim vertex adjacent to $t^+$ in $\TS^-(\Ho)$.
\end{proof}

\begin{lemma}\label{lm:003_2}
Suppose that $s\in V_s(\Ho)$ has exactly one fat neighbor in $\Ho$.
Then the following statements hold:
\begin{enumerate}
\item
$|N_{{\TS}^-(\Ho)}(s)|\leq4$, and if equality holds, then
$\TS^-(\Ho)$ is isomorphic to the
graph $\tilde{\D}_4$,
\item
if $|N_{{\TS}^-(\Ho)}(s)|=3$, then 
two of the vertices in $N_{{\TS}^-(\Ho)}(s)$ have $s$ as their
unique neighbor in $\TS^-(\Ho)$.
\end{enumerate}
\end{lemma}
\begin{proof}
Let $\psi$ be the reduced representation of norm $3$ of $\Ho$.
Since $s$ has exactly one fat neighbor, 
$(\psi(s),\psi(s))=2$. This means that we may assume without loss
of generality 
$\psi(s)=\be_1+\be_2$.
Let $f$ be the unique fat neighbor of $s$.
If $t\in N_{\TS^-(\Ho)}(s)$, then $t$ is adjacent to $f$, hence
\begin{equation}\label{pt}
\psi(t)\in\{-\be_1,-\be_2\}\cup\{-\be_i\pm\be_j\mid
1\leq i\leq 2<j\leq n\}.
\end{equation}
If $t,t'\in N_{\TS^-(\Ho)}(s)$ are distinct, then
\begin{align*}
1&\geq(\phi(t),\phi(t'))
\\&\geq
(\psi(t)+\be_f,\psi(t')+\be_f)
\nexteq
(\psi(t),\psi(t'))+1,
\end{align*}
Thus we have $(\psi(t),\psi(t'))\leq0$. 
If $|N_{\TS^-(\Ho)}(s)|\geq3$, then by (\ref{pt}), we may
assume without loss of generality that
there exists $t\in N_{\TS^-(\Ho)}(s)$ with
$\psi(t)=-\be_1+\be_3$. Then $\psi(N_{\TS^-(\Ho)}(s)\setminus\{t\})$ 
is contained in
\[
\{-\be_2-\be_3\},\{-\be_2,-\be_1-\be_3\},\text{ or }
\{-\be_1-\be_3\}\cup\{-\be_2\pm\be_j\}
\]
for some $j$ with $3\leq j\leq n$. Thus
$|N_{\TS^-(\Ho)}(s)|\leq4$, and equality 
holds only if
\[
\psi(N_{\TS^-(\Ho)}(s))=
\{-\be_1\pm\be_3,-\be_2\pm\be_j\}
\]
for some $j$ with $3\leq j\leq n$. In this case, Lemma~\ref{lm:004_1}
implies that each of the vertices in $N_{\TS^-(\Ho)}(s)$ has 
$s$ as a unique neighbor. This means that $\TS^-(\Ho)$ contains
a subgraph isomorphic to the graph $\tilde{\D}_4$ as a connected
component. Since $\TS^-(\Ho)$ is connected by Proposition~\ref{p:001}, 
we have the desired result.

To prove (ii), suppose $|N_{\TS^-(\Ho)}(s)|=3$. Then
we may assume without loss of generality that
$N_{\TS^-(\Ho)}(s)=\{t,t^+,t^-\}$, where
$\psi(t^\pm)=-\be_2\pm\be_4$. 
Then by Lemma~\ref{lm:004_1}, the vertices $t^\pm$ have
$s$ as their unique neighbor in $\TS^-(\Ho)$.
\end{proof}

\begin{theorem}\label{thm:001}
Let $\Ho$ be a fat indecomposable \sat{(-3)}\ Hoffman graph 
such that 
$\redlam (\Ho,3)$ is isomorphic to a sublattice of the standard lattice
$\zz^n$.
Then $\TS^-(\Ho)$ is a connected graph which is isomorphic to
	$\A_m,\D_m,\tilde{\A}_m$ or $\tilde{\D}_m$ 
	for some positive integer $m$.
\end{theorem}
\begin{proof}
\label{proofofthm:001}
From Proposition~\ref{p:001},
	${\TS}^-(\Ho)$ is connected.
First we suppose that
the maximum degree of ${\TS}^-(\Ho)$ is at most $2$.
Then ${\TS}^-(\Ho)\cong \tilde{\A}_m$ or ${\TS}^-(\Ho)\cong \A_m$
for some positive integer $m$.

Next we suppose that
	the degree of some vertex $s$ in ${\TS}^-(\Ho)$ is at least $3$.
From Lemma~\ref{lm:003_1}(ii) and Lemma~\ref{lm:003_2}(i),
	$\degr_{{\TS}^-(\Ho)}(s)\le 4$,
	and $\TS^-(\Ho)\cong\tilde{\D}_4$ if $\degr_{\TS^-(\Ho)}(s)= 4$.
Thus, for the remainder of this proof,
	we suppose that the maximum degree of ${\TS}^-(\Ho)$ is $3$ and
	$\degr_{\TS^-(\Ho)}(s)=3$.

It follows from Lemma~\ref{lem:sat} that if
	$\Ho$ has a subgraph isomorphic to 
$\Ho^{(3)}$, then $\Ho\cong\Ho^{(3)}$,
	in which case $\TS^-(\Ho)$ consists of a single vertex, and the assertion holds.
Hence it remains to consider two cases:
	$s$ is adjacent to exactly two fat vertices,
	and $s$ is adjacent to exactly one fat vertex.
In either cases, by Lemma~\ref{lm:003_1}(iii) or Lemma~\ref{lm:003_2}(ii),
$s$ has at most one neighbor $t$ with degree greater than $1$.
Thus, the only way to extend this graph is
	by adding a slim neighbor adjacent to $t$.
	We can continue this process, but once we encounter a vertex of degree
	$3$, then the process stops by Lemma~\ref{lm:003_1}(iii) or Lemma~\ref{lm:003_2}(ii). 
	Thus, $\TS^-(\Ho)$ is isomorphic to one of the graphs $\D_m$ or $\tilde{\D}_m$.
\end{proof}

\begin{exa}\label{exa:An}
Let $n_1,\dots,n_k$ be positive integers satisfying
$n_i\geq2$ for $1< i< k$.
Set $m_j=\sum_{i=1}^j n_i$ and $\ell_j=m_j-j$ for
$j=0,1,\dots,k$.
Let $\Ho$ be the Hoffman graph with
$V_s(\Ho)=\{v_i\mid i=0,1,\dots,m_k\}$, 
$V_f(\Ho)=\{f_j\mid j=0,1,\dots,k+1\}$, and
\begin{align*}
E(\Ho)=&\{\{v_i,v_{i'}\}\mid 1\leq j\leq k,\;m_{j-1}< i+1<i'\leq m_j\}
\\&\cup\{\{v_{m_j-1},v_{m_j+1}\}\mid 1<j<k\}
\\&\cup\{\{f_j,v_i\}\mid 1\leq j<k,\;m_{j-1}\leq i\leq m_j\}
\cup\{\{f_0,v_0\},\{f_{k+1},v_{m_k}\}\}.
\end{align*}
Then $\Ho$ is a fat Hoffman graph with
smallest eigenvalue at least $-3$, and $\TS^-(\Ho)$ is isomorphic
to the Dynkin graph $\A_{m_k+1}$.
Indeed, $\Ho$ has a reduced representation $\psi$ of norm $3$ defined by
\[
\psi(v_i)=\begin{cases}
(-1)^j\be_{\ell_j}&\text{if $i=m_j$, $0\leq j\leq k$,}\\
(-1)^j(\be_{i-j}-\be_{i-j-1})&\text{if $m_j<i<m_{j+1}$, $0\leq j< k$.}
\end{cases}
\]
Moreover, $\Ho$ is \sat{(-3)}. Indeed, suppose not, and let $\tilde{\Ho}$
be a Hoffman graph obtained by attaching a new fat vertex 
$f$ to $\Ho$, and let $\tilde{\psi}$ be a reduced representation of norm
$3$ of $\tilde{\Ho}$. If $f$ is adjacent to $v_{m_j}$ for some
$j\in\{0,1,\dots,k\}$, then $v_{m_j}$ has three fat neighbors in $\tilde{\Ho}$,
hence $\tilde{\psi}(v_{m_j})=0$. This is absurd, since
$(\tilde{\psi}(v_{m_j}),\tilde{\psi}(v_{m_j\pm1}))=-1$.
If $f$ is adjacent to $v_i$ with $m_{j-1}<i<m_j$, then
$\|\tilde{\psi}(v_i)\|=1$. Since
$(\tilde{\psi}(v_{i-1}),\tilde{\psi}(v_{i}))=
(\tilde{\psi}(v_{i+1}),\tilde{\psi}(v_{i}))=-1$ while
$(\tilde{\psi}(v_{i-1}),\tilde{\psi}(v_{i+1}))=0$, we may assume
$\tilde{\psi}(v_{i\pm1})=\be_1\pm\be_2$,
$\tilde{\psi}(v_{i})=-\be_1$.
Then $i+1<m_j$, and
\begin{align*}
0&=(\tilde{\psi}(v_{i+2}),\tilde{\psi}(v_{i-1}))
\nexteq
(\tilde{\psi}(v_{i+2}),\be_1-\be_2)
\nexteq
(\tilde{\psi}(v_{i+2}),2\be_1-(\be_1+\be_2))
\nexteq
-2(\tilde{\psi}(v_{i+2}),\tilde{\psi}(v_i))
-(\tilde{\psi}(v_{i+2}),\tilde{\psi}(v_{i+1}))
\nexteq
1,
\end{align*}
which is absurd.

We note that the graph $\TS^+(\Ho)$ has the following edges:
\[
\{\{v_{m_j-1},v_{m_j+1}\}\mid 1<j<k\}.
\]
\end{exa}

\begin{exa}\label{exa:A5}
Let $\Ho$ be the Hoffman graph constructed in 
Example~\ref{exa:An} by setting
$n_1=1$, $n_2=2$, and $n_3=1$.
Let $\Ho_0$ (resp.\ $\Ho_1$) be the Hoffman graph obtained from
$\Ho$ by identifying the fat vertices $f_4$ and $f_0$
(resp.\ $f_4$ and $f_1$), and adding edges $\{v_0,v_2\},\{v_2,v_4\}$.
Then $\Ho_0$ and $\Ho_1$ are fat \sat{(-3)}\ Hoffman graphs
and $\TS^-(\Ho_i)$ is isomorphic to the Dynkin graph $\A_5$
for $i=0,1$. We note that $\TS^+(\Ho_i)$ has two
edges $\{v_0,v_2\},\{v_2,v_4\}$.
\end{exa}
Examples~\ref{exa:An} and \ref{exa:A5} indicate that
$\Ho$ is not determined by $\TS^\pm(\Ho)$. We plan to
discuss the classification of fat indecomposable \sat{(-3)}\ 
Hoffman graphs with prescribed special graph in the
subsequent papers.


\begin{thebibliography}{9}
\bibitem{new} D.~Cvetkovi\'{c}, P.~Rowlinson, S.~K.~Simi\'{c},
		Spectral Generalizations of Line Graphs --- On Graphs with Least Eigenvalue $-2$,
		Cambridge Univ. Press, 2004.
\bibitem{Eb} W.~Ebeling, Lattices and Codes, Vieweg, 2nd ed.
Friedr. Vieweg \& Sohn, Braunschweig, 2002.
\bibitem{hoffman0} A.~J.~Hoffman, On graphs whose least eigenvalue exceeds $-1-\sqrt{2}$,
		Linear Algebra Appl. 16 (1977), 153--165.
\bibitem{paperI} T.~Taniguchi, On graphs with the smallest eigenvalue
		at least $-1-\sqrt{2}$,
		part I,
		Ars Math. Contemp. 1 (2008),
		no.1,
		81--98.
\bibitem{hlg} R.~Woo and A.~Neumaier,
	On graphs whose smallest eigenvalue is at least $-1-\sqrt{2}$,
		Linear Algebra Appl. 226-228 (1995), 577--591 .
\bibitem{Yu} H. Yu,
On the limit points of the smallest eigenvalues of regular graphs,
Designs, Codes Cryptogr.,
published online, DOI: 10.1007/s10623-011-9575-0.
\end{thebibliography}
\end{document}